\documentclass[10pt,a4paper]{amsart}

\usepackage{amssymb}
\usepackage{amsthm}
\usepackage{amsfonts}
\usepackage{microtype}
\usepackage{mathrsfs}
\usepackage{graphicx}
\usepackage{color}
\usepackage{latexsym}
\usepackage{rotating}
\usepackage{xspace}
\usepackage[all]{xy}
\usepackage{longtable}
\usepackage{leftidx}
\usepackage{mathtools}
\usepackage{titlesec}
\usepackage{tikz}
\usetikzlibrary{calc}
\usetikzlibrary{arrows}
\usetikzlibrary{decorations.markings}
\usepackage{geometry}

\newtheorem{thm}{Theorem}[section]
\newtheorem{cor}[thm]{Corollary}
\newtheorem{lem}[thm]{Lemma}

\newcommand{\Ann}{\mbox{Ann}\,}

\newcommand{\Ker}{\mbox{Ker}\,}

\newcommand{\h}{\mbox{ht}\,}

\newcommand{\E}{\mbox{E}}
\newcommand{\uhom}{{\mathbf R}\Hom}
\newcommand{\utp}{\otimes^{\mathbf L}}
\newcommand{\ugamma}{{\mathbf R}\Gamma}
\renewcommand{\H}{\mbox{H}}
\newcommand{\V}{\mbox{V}}

\newcommand{\lo}{\longrightarrow}

\newcommand{\fa}{\mathfrak{a}}
\newcommand{\fb}{\mathfrak{b}}

\newcommand{\fm}{\mathfrak{m}}
\newcommand{\fp}{\mathfrak{p}}
\newcommand{\fq}{\mathfrak{q}}

\newcommand{\fc}{\mathfrak{c}}

\renewcommand{\Im}{\mbox{Im}\,}

\newcommand{\Min}{\mbox{Min}\,}

\def\Ext{\operatorname{\mathsf{Ext}}}
\def\depth{\operatorname{\mathsf{depth}}}
\def\Hom{\operatorname{\mathsf{Hom}}}
\def\dim{\operatorname{\mathsf{dim}}}

\DeclareMathOperator{\Supp}{Supp}
\DeclareMathOperator{\Spec}{Spec}
\DeclareMathOperator{\Ass}{Ass}

\makeatletter
\newcommand{\holim@}[2]{%
  \vtop{\m@th\ialign{##\cr
    \hfil$#1\operator@font holim$\hfil\cr
    \noalign{\nointerlineskip\kern1.5\ex@}#2\cr
    \noalign{\nointerlineskip\kern-\ex@}\cr}}%
}
\newcommand{\holim}{%
  \mathop{\mathpalette\holim@{\rightarrowfill@\textstyle}}\nmlimits@
}
\makeatother

\makeatletter
\def\@secnumfont{\bfseries}
\def\section{\@startsection{section}{1}%
  \z@{.7\linespacing\@plus\linespacing}{.5\linespacing}%
  {\normalfont\Large\bfseries\filcenter}}
\def\subsection{\@startsection{subsection}{2}%
  \z@{.5\linespacing\@plus.7\linespacing}{-.5em}%
  {\normalfont\large\bfseries}}
\makeatother

\begin{document}

\title[The derived category analogues of ...]
{The derived category analogues of Faltings' Local-global Principle and Annihilator Theorems}

\author[K. Divaani-Aazar]{Kamran Divaani-Aazar}
\author[M. Rahro Zargar]{Majid Rahro Zargar}

\address{K. Divaani-Aazar, Department of Mathematics, Alzahra University, Vanak, Post Code
19834, Tehran, Iran-and-School of Mathematics, Institute for Research in Fundamental Sciences
(IPM), P.O. Box 19395-5746, Tehran, Iran.}
\email{kdivaani@ipm.ir}

\address{Majid Rahro Zargar, Department of Engineering Sciences, Faculty of Advanced Technologies,
University of Mohaghegh Ardabili, Namin, Iran-and-School of Mathematics, Institute for Research in
Fundamental Sciences (IPM), P.O. Box: 19395-5746, Tehran, Iran.}
\email{zargar9077@gmail.com}

\keywords{Adjusted depth; derived category; finiteness dimension; local cohomology; specialization closed set.\\
The research of the first author is supported by a grant from IPM (no. 93130212).}

\begin{abstract}
Let $\mathcal{Z}$ be a specialization closed subset of $\Spec R$ and $X$ a homologically left-bounded complex
with finitely generated homologies. We establish Faltings' Local-global Principle and Annihilator Theorems for
the local cohomology modules {{$\H_{\mathcal{Z}}^i(X).$ }} Our versions contain variations of results already
known on these theorems.
\end{abstract}

\maketitle

\section{Introduction}

Throughout, $R$ is a commutative Noetherian ring with identity. Let $\fa$ be an ideal of $R$ and $M$ a finitely generated
$R$-module. The finiteness dimension of $M$ relative to $\fa$, $f_{\fa}(M)$, is defined as the infimum of the integers $i$
such that $\H^i_{\fa}(M)$ is not finitely generated. Let $r$ be a positive integer. It is known that $\H^i_{\fa}(M)$ is
finitely generated for all $i<r$ if and only if $\fa^n\H^i_{\fa}(M)=0$ for some positive integer $n$ and all $i<r$.
Faltings' Local-global Principle Theorem \cite[Satz 1]{Fa1} asserts that the $R$-module $\H^i_{\fa}(M)$ is finitely generated
for all $i<r$ if and only if the $R_{\fp}$-module $\H^i_{\fa R_{\fp}}(M_{\fp})$ is finitely generated for all $i<r$ and for all
$\fp\in \Spec R$. Thus, $$f_{\fa}(M)=\inf \left\{i\in \mathbb{N}_0|\ \fa\nsubseteq \sqrt{(0:_R\H^i_{\fa}(M))} \right\}=\inf
\left\{f_{\fa R_{\fp}}(M_{\fp})|\ \fp\in \Spec R \right\}.$$

Now, let $\fb$ be a second ideal of $R$ such that $\fb\subseteq \fa$. The $\fb$-finiteness dimension of $M$ relative to $\fa$
is defined by $$f^{\fb}_{\fa}(M):=\inf \left\{i\in \mathbb{N}_0|\ \fb\nsubseteqq \sqrt{(0:_R\H^i_{\fa}(M))}\right\}.$$ It is
natural to ask whether Faltings' Local-global Principle generalizes for the pair $\fb\subseteq \fa$. In other words, does
$$f_{\fa}^{\fb}(M)=\inf\left\{f_{\fa R_{\fp}}^{\fb R_{\fp}}(M_{\fp})|\ \fp \in \Spec R \right\}?$$

In \cite[Corollary]{Ra}, Raghavan deduced from Faltings' Annihilator Theorem that if $R$ is a homomorphic image of a regular
ring, then the Local-global Principle holds for the pair $\fb\subseteq \fa$. The $\fb$-minimum $\fa$-adjusted depth of $M$
is defied by $$\lambda_{\fa}^{\fb}(M):=\inf \left\{\depth M_{\fp}+\h \left(\frac{\fa+\fp}{\fp}\right)|\ \fp\in \Spec R-\V(\fb)
\right\}.$$ It is always the case that $f^{\fb}_{\fa}(M)\leq \lambda_{\fa}^{\fb}(M)$. Faltings' Annihilator Theorem
\cite[Satz 1]{Fa2} states that if $R$ is a homomorphic image of a regular ring,  then $f^{\fb}_{\fa}(M)=\lambda_{\fa}^{\fb}(M)$.

In the literature, there are many generalizations of Faltings' Local-global Principle and Annihilator Theorems for ordinary local
cohomology and also for some of its generalizations; see e.g. \cite{AKS}, \cite{BRS}, \cite{Ka}, \cite{KS}, \cite{KYA} and \cite{Ra}.

It is known that all the generalizations $H_{\mathcal{Z}}^{i}(M)$, $H_{\mathcal{\mathfrak{a},\mathfrak{b}}}^{i}(M)$, and
$H_{\mathcal{\mathfrak{a}}}^{i}(M,N)$ of the local cohomology module $H_{\mathcal{\mathfrak{a}}}^{i}(M)$ of an $R$-module $M$,
are special cases of the local cohomology module $H_{\mathcal{Z}}^{i}(X)$ of a complex $X$ with support in a specialization
closed subset $\mathcal{Z}$ of $\Spec R$. For the definitions of $H_{\mathcal{\mathfrak{a},\mathfrak{b}}}^{i}(M)$ and
$H_{\mathcal{\mathfrak{a}}}^{i}(M,N)$, we refer the reader to \cite{TYY} and \cite{He}. Also, Yoshino and Yoshizawa
\cite[Theorem 2.10]{YY} have shown that for every abstract local cohomology functor $\delta$ from the category of homologically left
bounded complexes of $R$-modules to itself, there is a specialization closed subset $\mathcal{Z}$ of $\Spec R$ such that $\delta
\cong \ugamma_{\mathcal{Z}} $. Therefore, any established result on $H_{\mathcal{Z}}^{i}(X)$ encompasses all the previously known
results on each of these local cohomology modules.

Our aim in this paper is to establish Faltings' Local-global Principle and Annihilator Theorems for the local cohomology modules
{{$\H_{\mathcal{Z}}^i(X).$ }} More precisely, we prove the following theorem; see Theorems \ref{G} and \ref{M}, and Corollaries
\ref{H} and \ref{P}. To state it, we first need to fix some notation.

Let $\mathcal{Z}\subseteq \mathcal{Y}$ be two specialization closed subsets of $\Spec R$ and $X$ a homologically left
bounded complex with finitely generated homologies.  Set $$f_{\tiny{\mathcal{Z}}}^{\tiny{\mathcal{Y}}}(X):=\inf \left\{i\in
\mathbb{Z}|\ \fc \H_{\mathcal{Z}}^i(X)\neq 0\ \text{for all ideals}\ \fc \ \text{of} \ R \  \text{with} \ \V(\fc)\subseteq
\mathcal{Y}\right\}$$ and $$\lambda_{\mathcal{Z}}^{\mathcal{Y}}(X):=\inf\left\{\depth_{R_{\fq}}X_{\fq}+\h\frac{\fp}{\fq}|\ \fq\notin
\mathcal{Y} \  \ \text{and} \ \ \fp\in\mathcal{Z}\cap \V(\fq)\right\}.$$ Abbreviate $f_{\mathcal{Z}}^{\mathcal{Z}}(X)$ and
$\lambda_{\mathcal{Z}}^{\mathcal{Z}}(X),$ by $f_{\mathcal{Z}}(X)$ and $\lambda_{\mathcal{Z}}(X)$, respectively. Note that
\emph{$f_{\fa}^{\fb}(M)=f_{\tiny{{\V(\fa)}}}^{\tiny{\V(\fb)}}(M)$} and \emph{$\lambda_{\fa}^{\fb}(M)=\lambda_{\tiny{{\V(\fa)}}}^{
\tiny{\V(\fb)}}(M).$}

\begin{thm}\label{1}
Let $\mathcal{Z}\subseteq \mathcal{Y}$ be two specialization closed subsets of $\Spec R$ and $X$ a homologically left-bounded
complex with finitely generated homologies. Then the following statements hold.
\begin{itemize}
\item[(i)] \emph{{$f_{\mathcal{Z}}(X)=\inf \left\{f_{\mathcal{Z_{\fp}}}(X_{\fp})|\  \fp \in \Spec R\right\}=\inf \left\{i
\in \mathbb{Z}|\ \H_{\mathcal{Z}}^i(X) \  \text{is not finitely generated}\right\}.$}}
\item[(ii)] Assume that $X$ is homologically bounded. Then $f_{\mathcal{Z}}^{\mathcal{Y}}(X)
\leq \lambda_{\mathcal{Z}}^{\mathcal{Y}}(X)$.
\item[(iii)] Assume that $R$ is a homomorphic image of a finite-dimensional Gorenstein ring and $X$ is homologically bounded.
Then $f_{\mathcal{Z}}^{\mathcal{Y}}(X)=\lambda_{\mathcal{Z}}^{\mathcal{Y}}(X)$ and $$f_{\mathcal{Z}}^{\mathcal{Y}}(X)
=\inf \left\{f_{\mathcal{Z_{\fp}}}^{\mathcal{Y_{\fp}}}(X_{\fp})|\   \fp \in \Spec R\right\}.$$
\end{itemize}
\end{thm}

\section{Prerequisites}

The derived category of $R$-modules is denoted by $\mathrm{D}(R)$. Simply put, an object in $\mathcal{D}(R)$ is an $R$-complex $X$ displayed
in the standard homological style $$X= \cdots \rightarrow X_{i+1} \xrightarrow {\partial^{X}_{i+1}} X_{i} \xrightarrow {\partial^{X}_{i}} X_{i-1}
\rightarrow \cdots.$$ We use the symbol $\simeq$ for denoting isomorphisms in $\mathrm{D}(R)$. We denote the full subcategory of homologically
left-bounded complexes by $\mathrm{D}_{\sqsubset}(R)$. Also, we denote the full subcategory of complexes with finitely generated homology modules
that are homologically bounded (resp. homologically left-bounded) by $\mathrm{D}_{\Box}^f(R)$ (resp. $\mathrm{D}_{\sqsubset}^f(R)$).  Given an
$R$-complex $X$, the standard notion $$\sup X=\sup \left\{i \in \mathbb{Z}| \ \H_{i}(X) \neq 0 \right\}$$
is frequently used, with the convention that $\sup \emptyset=-\infty$.

Let $\fa$ be an ideal of $R$ and $X\in \mathrm{D}_{\sqsubset}(R)$. A subset $\mathcal{Z}$ of $\Spec R$ is said
to be {\it specialization closed} if $\V(\fp)\subseteq \mathcal{Z}$ for all $\fp \in \mathcal{Z}$. For every
$R$-module $M$, set $\Gamma_{\mathcal{Z}}(M):=\left\{x\in M|~\Supp_RRx\subseteq \mathcal{Z}\right\}.$ The right
derived functor
of the functor $\Gamma_{\mathcal{Z}}(-)$ in $\mathrm{D}(R)$, $\ugamma_{
\mathcal{Z}}(X)$, exists and is defined by ${\bf R} \Gamma_{\mathcal{Z}}(X):=\Gamma_{\mathcal{Z}}(I)$, where $I$
is any injective resolution of $X$.  Also, for every integer $i$, the $i$-th local cohomology module of $X$ with
respect to $\mathcal{Z}$ is defined by $\H_{\mathcal{Z}}^i(X):=\H_{-i}({\bf R}\Gamma_{\mathcal{Z}}(X))$. To
comply with the usual notation, for $\mathcal{Z}:=\V(\fa)$, we denote $\ugamma_{\mathcal{Z}} (-)$ and $\H_{
\mathcal{Z}}^i(-)$ by ${\bf R}\Gamma_{{\fa}}(-)$ and $\H_{\fa}^i(-)$, respectively. Denote the set of all ideals
$\fb$ of $R$ such that $\V(\fb)\subseteq \mathcal{Z}$ by $F(\mathcal{Z})$. Since for every $R$-module $M,$ $\Gamma_{
\mathcal{Z}}(M)=\bigcup_{\fb\in F(\mathcal{Z})}\Gamma_{\fb}(M)$, for every integer $i$, one can easily check that
$$\H_{\mathcal{Z}}^i(X)\cong \underset{\fb\in F(\mathcal{Z})}\varinjlim\H_{\fb}^i(X).$$

Recall that $\Supp_RX:=\underset{l\in \mathbb{Z}}\bigcup\Supp_R\H_l(X)$ and $$\depth(\fa,X):=-\sup {\bf R}\Hom_R(R/\fa,X).$$ By
\cite[Theorem 6.2]{Iy}, it is known that $$\depth(\fa,X)=\inf\left\{i\in \mathbb{Z}|\  \H_{\fa}^i(X)\neq 0\right\}.$$ When $R$
is local with maximal ideal $\fm$, $\depth(\fm,X)$ is simply denoted by $\depth_RX$.  For every prime ideal $\fp$ of $R$
and every integer $i$, the i-th Bass number $\mu^i(\fp,X)$ is defined to be the dimension of the $R_{\fp}/\fp R_{\fp}$-vector
space $\H_{-i}(\uhom_{R_{\fp}}(R_{\fp}/\fp R_{\fp},X_{\fp}))$.

\section{Local-global Principle Theorem}

The following easy observation will be very useful in the rest of the paper.

\begin{lem}\label{A} Let $\mathcal{Z}$ be a specialization closed subset of $\Spec R$ and $X\in\mathrm{D}_{
\sqsubset}(R).$ Then the following statements hold.
\begin{itemize}
\item[(i)] \emph{$\H_{\mathcal{Z}}^i(X)=0$} for all $i<-\sup X$ and \emph{$\H_{\mathcal{Z}}^{-\sup X}(X)$} is a
finitely generated $R$-module whenever $X\in\mathrm{D}_{\sqsubset}^f(R)$.
\item[(ii)]  \emph{$\Supp_R\H_{\mathcal{Z}}^i(X)\subseteq \mathcal{Z}$} for every integer $i$.
\item[(iii)] If \emph{$\Supp_RX\subseteq  \mathcal{Z}$,} then \emph{$\ugamma_{\mathcal{Z}} (X)\simeq X$}.
\end{itemize}
\end{lem}

\begin{proof} (i) Let $s:=\sup X$. By \cite[Theorem A.3.2 (I)]{Ch}, $X$ possesses an injective resolution $I$ such that
$I_i=0$ for all $i>s$. As for every integer $i$, one has $\H_{\mathcal{Z}}^i(X)=\H_{-i}(\Gamma_{\mathcal{Z}} (I))$, it
follows that $\H_{\mathcal{Z}}^i(X)=0$ for all $i<-s$ and that $\H_{\mathcal{Z}}^{-s}(X)$ is a submodule of $\H_{s}(I)$.
If $X\in\mathrm{D}_{\sqsubset}^f(R)$, then $\H_{s}(X)$ is finitely generated, and so $\H_{\mathcal{Z}}^{-s}(X)$ is finitely
generated too.

(ii) the proof is easy and we leave it to the reader.

(iii) Assume that \emph{$\Supp_RX\subseteq  \mathcal{Z}$.} For every prime ideal $\fp$  of $R$, one can check that
$$\Gamma_{\mathcal{Z}}(\E(R/\fp))=\begin{cases} \E(R/\fp) & ,\fp\in \mathcal{Z}\\ 0& ,\fp\notin \mathcal{Z}.
\end{cases}$$
By \cite[Lemma 2.3 (a) and Proposition 3.18]{Fo}, $X$ possesses an injective resolution $I$ such that $$I_i \cong \underset{\fp\in
\Supp_RX}\bigoplus\E(R/\fp)^{(\mu^i(\fp,X))}$$ for all integers $i$. Thus, $\Gamma_{\mathcal{Z}}(I_i)\cong I_i$ for all
integers $i$, and so $$\ugamma_{\mathcal{Z}} (X)\simeq \Gamma_{\mathcal{Z}}(I)\simeq I\simeq X.$$
\end{proof}

The following result plays an essential role in the proof of the derived category analogue of Falting's Local-global
Principle Theorem.

\begin{lem}\label{C} Let $\mathcal{Z}$ be a specialization closed  subset of $\Spec R$ and $X\in\mathrm{D}_{
\sqsubset}^f(R).$ Let $t$ be an integer such that \emph{$\H_{\mathcal{Z}}^i(X)$} is a finitely generated $R$-module
for all $i<t$. Then for every $\fa\in F(\mathcal{Z})$, the $R$-module \emph{$\Hom_R(R/\fa,\H_{\mathcal{Z}}^t(X))$} is
finitely generated.
\end{lem}

\begin{proof} By Lemma \ref{A} (i), $\H_{\mathcal{Z}}^i(X)=0$ for all $i<-\sup X$ and $\H_{\mathcal{Z}}^{-\sup X}(X)$ is
finitely generated. So, we may assume that $-\sup X<t$. Set $T:=\Sigma^{-\sup X}X$ and note that $\H_{\mathcal{Z}}^i(X)
\cong \H_{\mathcal{Z}}^{i+\sup X}(T)$ for all integers $i$. So, by replacing $X$ with $T$, we may and do assume that
$\sup X=0$ and $0<t$. Then, there exists an injective resolution $I$ of $X$ such that $I_{l}=0$ for all $l>0$.

Now, let $\fa\in F(\mathcal{Z})$ and $P$ be a projective resolution of $R/\fa$. Set $M_{p,q}:=\Hom_{R}(P_{-p}, \Gamma_{
\mathcal{Z}}(I_{q}))$. Hence $\mathcal{M}:=\left\{M_{p,q}\right\}$ is a third quadrant bicomplex, and so the complex
$\Hom_{R}(P,\Gamma_{\mathcal{Z}}(I))$ is the total complex of $\mathcal{M}$.

For every $R$-module $M$, one has $\Gamma_{\fa}(\Gamma_{\mathcal{Z}}(M))=\Gamma_{\fa}(M)$. So, the two complexes
$\Gamma_{\fa}(\Gamma_{\mathcal{Z}}(I))$ and $\Gamma_{\fa}(I)$ are the same. By \cite[Propositon 3.2.2]{Li}, for any two
complexes $X_1, X_2\in \mathrm{D}_{\sqsubset}(R)$ with $\Supp_RX_1\subseteq \V(\fa),$  one has an isomorphism $$\uhom_{R}(X_1,X_2)
\simeq \uhom_{R}(X_1,\ugamma_{\fa}(X_2))$$ in $\mathrm{D}(R)$. This yields $\dag$ and $\ddag$ in the following display
of isomorphisms in $\mathrm{D}(R)$:

\[\begin{array}{rl}
\Hom_{R}(P,\Gamma_{\mathcal{Z}}(I))&\simeq \uhom_{R}(R/\fa,\Gamma_{\mathcal{Z}}(I))\\
&\overset{\dag}\simeq \uhom_{R}(R/\fa,\ugamma_{\fa}(\Gamma_{\mathcal{Z}}(I)))\\
&\simeq \uhom_{R}(R/\fa,\Gamma_{\fa}(\Gamma_{\mathcal{Z}}(I)))\\
&\simeq \uhom_{R}(R/\fa,\Gamma_{\fa}(I))\\
&\simeq \uhom_{R}(R/\fa,\ugamma_{\fa}(I))\\
&\overset{\ddag}\simeq \uhom_{R}(R/\fa,I).
\end{array}\]
Thus, there is a first quadrant spectral sequence $$\E_{2}^{p,q}:=\Ext_{R}^p(R/\fa,\H_{\mathcal{Z}}^q(X))\underset{p}
\Longrightarrow\Ext_{R}^{p+q}(R/\fa,X).$$ Note that $\Ext_{R}^{i}(R/\fa,X)$ is finitely generated for all integers
$i$. For each $r\geq2$, let $$Z_{r}^{0,t}:=\Ker(\E_{r}^{0,t}\longrightarrow\E_{r}^{r,t+1-r})$$ and $$B_{r}^{0,t}:=
\Im(\E_{r}^{-r,t+r-1}\longrightarrow\E_{r}^{0,t}).$$ As $\E_{r}^{p,q}$ is a subquotient of $\E_{2}^{p,q}$, it follows
that $\E_{r}^{-r,t+r-1}=0$ and $\E_{r}^{r,t+1-r}$ is finitely generated. Hence $$\E_{r+1}^{0,t}=\frac{Z_{r}^{0,t}}{B_{r}^{0,t}}
\cong Z_{r}^{0,t}$$ and $\E_{r}^{0,t}/Z_{r}^{0,t}$ is finitely generated. Thus $\E_{r}^{0,t}$ is a finitely generated
$R$-module if and only if $\E_{r+1}^{0,t}$ is a finitely generated $R$-module. Now, we claim that $\E_{r}^{0,t}$ is
a finitely generated $R$-module for all $r\geq 2$. To this end, we use descending induction on $r$. Let $r\geq t+2$.
Then one can use the fact that $\sup X=0$ to deduce that $\E_{r}^{-r,t+r-1}=\E_{r}^{r,t+1-r}=0$, and so $\E_{t+2}^{0,t}
\cong\ldots\cong\E_{\infty}^{0,t}$. Now, consider the following filtration $$\left\{0\right\}=\Psi_{t+1}\H^{t}\subseteq \Psi_{t}
\H^{t}\subseteq \cdots\subseteq\Psi_{1}\H^t\subseteq \Psi_{0}\H^t=\H^t,$$ where $\H^t:=\Ext_{R}^t(R/\fa,X)$ and
$\E^{p,t-p}_{\infty}=\frac{\Psi_{p}\H^t}{\Psi_{p+1}\H^t}$, to see that $\E_{t+2}^{0,t}$ is finitely generated. Suppose
that the result has been proved for all $2<r\leq t+2$. Then we want to show that the result holds for $r-1$. Notice
that $2\leq r-1$, and so by the above argument and inductive hypothesis one can see that $\E_{r-1}^{0,t}$ is finitely
generated. It therefore follows that $\Hom_{R}(R/\fa,\H_{\mathcal{Z}}^t(X))$ is a finitely generated $R$-module.
\end{proof}

Next, we record the following immediate consequence.

\begin{cor}\label{D} Let $\mathcal{Z}$ be a specialization closed subset of $\Spec R$ and $X\in\mathrm{D}_{\sqsubset}^f(R)$.
Then for every integer $t$, the following statements are equivalent:
\begin{itemize}
\item[(i)]\emph{{$\H_{\mathcal{Z}}^i(X)$ is a finitely generated $R$-module for all $i<t$.}}
\item[(ii)]\emph{{There exists an ideal $\fa \in F(\mathcal{Z})$ such that $\fa\H_{\mathcal{Z}}^i(X)=0$ for all $i<t$.}}
\end{itemize}
\end{cor}

\begin{proof} (i)$\Rightarrow$(ii) For each $i<t$, set $\fa_{i}:=(0:_R \H_{\mathcal{Z}}^i(X))$ and note that Lemma \ref{A} (ii)
implies that $\fa_{i}\in F(\mathcal{Z})$. Now, the ideal $\fa:=\prod_{i=-\sup X}^{t-1}\fa_{i}$ belongs to $F(\mathcal{Z})$ and
$\fa \H_{\mathcal{Z}}^i(X)=0$ for all $i<t$.

(ii)$\Rightarrow$(i) We may and do assume that $t\geq 1-\sup X$ and proceed by induction on $t$. If $t=1-\sup X$, then by Lemma
\ref{A} (i) we see that $\H_{\mathcal{Z}}^{i}(X)$ is a finitely generated $R$-module for all $i<t$. Let $t>1-\sup X$ and suppose
that the result has been proved for $t-1$. Now by the induction hypothesis, $\H_{\mathcal{Z}}^i(X)$ is a finitely generated
$R$-module for all $i<t-1$, and so by Lemma \ref{C} the $R$-module $\Hom_R(R/\fa,\H_{\mathcal{Z}}^{t-1}(X))$ is finitely generated.
But, by our assumption $\fa \H_{\mathcal{Z}}^{t-1}(X)=0$, and so  $\Hom_R(R/\fa,\H_{\mathcal{Z}}^{t-1}(X))\cong \H_{\mathcal{Z}}^{
t-1}(X)$.
\end{proof}

Let $T$ be a second commutative Noetherian ring with identity and $f:R\lo T$ be a ring homomorphism. Let $\mathcal{Z}$
be a specialization closed subset of $\Spec R$. Then it is easy to check that $\mathcal{Z}^f:=\left\{\fq \in \Spec T|
\  f^{-1}(\fq)\in \mathcal{Z}\right\}$ is a specialization closed subset of $\Spec T$.

\begin{lem}\label{B} Let $f:R\lo T$ be a ring homomorphism and $\mathcal{Z}$ a specialization closed subset of
$\Spec R$. Let $X\in \mathrm{D}_{\sqsubset}(R)$ and $Y\in \mathrm{D}_{\sqsubset}(T)$. Then the following statements hold.
\begin{itemize}
\item[(i)] There is a natural $R$-isomorphism \emph{$\H_{\mathcal{Z}}^i(Y)\cong \H_{\mathcal{Z}^f}^i(Y)$} for all integers $i$.
\item[(ii)] Suppose that $T$ is flat as an $R$-module.  There is a natural $T$-isomorphism \emph{$\H_{\mathcal{Z}}^i(X)\otimes_RT
\cong \H_{\mathcal{Z}^f}^i(X\otimes_RT)$} for all integers $i$.
\end{itemize}
\end{lem}

\begin{proof} Set $\Omega:=\left\{\fa T|\ \fa\in F(\mathcal{Z})\right\}$. Then, it is easy to see that $\Omega\subseteq F(\mathcal{Z}^f)$.
Let $\widetilde{\fb}\in F(\mathcal{Z}^f)$ and set $\fb:=f^{-1}(\widetilde{\fb})$. We claim that $\fb\in F(\mathcal{Z})$. To this end, it
is enough to show that every  minimal element $\fp$ in $\V(\fb)$ belongs to $\mathcal{Z}$.

Let $\fp$ be a minimal element in  $\V(\fb)$ and  $\widetilde{\fb}=\bigcap_{i=1}^nQ_i$ be a minimal primary decomposition of $\widetilde{\fb}$
in $T$. Then $\sqrt{Q_i}\in \mathcal{Z}^f$ for all $i=1,\dots, n$. As $\fb=\bigcap_{i=1}^nf^{-1}(Q_i)$ is a primary decomposition
of $\fb$, it turns out that $\fp=f^{-1}(\sqrt{Q_j})$ for some $1\leq j\leq n$. So, $\fp\in \mathcal{Z}$. Hence $\fb\in F(\mathcal{Z})$,
and so $\fb T\in \Omega$. Clearly, $\fb T\subseteq \widetilde{\fb}$. Thus the two families of ideals $\Omega$ and $F(\mathcal{Z}^f)$ are cofinal.

(i) Let $\fa$ be an ideal of $R$. Then by \cite[Corollary 3.4.3]{Li}, there is an $R$-isomorphism $\H_{\fa}^i(Y)\cong \H_{\fa T}^i(Y)$ for
all integers $i$. Hence, (i) follows by the following display of $R$-isomorphisms

\[\begin{array}{rl}
\H_{\mathcal{Z}}^i(Y)&\cong \underset{{\fa \in F(\mathcal{Z})}}\varinjlim \H_{\fa}^i(Y)\\
&\cong \underset{{\fa\in F(\mathcal{Z})}}\varinjlim \H_{\fa T}^i(Y)\\
&\cong \  \ \underset{{{\widetilde{\fb}\in \Omega}}}\varinjlim \  \  \H_{\widetilde{\fb}}^i(Y)\\
&\cong \underset{{{\widetilde{\fb}\in F(\mathcal{Z}^f)}}}\varinjlim \H_{\widetilde{\fb}}^i(Y)\\
&\cong \ \ \H_{\mathcal{Z}^f}^i(Y).
\end{array}\]

(ii) In view of \cite[Corollary 3.4.4]{Li}, one has the third isomorphism in the following display of $T$-isomorphisms

\[\begin{array}{rl}
\H_{\mathcal{Z}}^i(X)\otimes_RT&\cong (\underset{{\fa\in F(\mathcal{Z})}}\varinjlim \H_{\fa}^i(X))\otimes_RT\\
&\cong \underset{{\fa\in F(\mathcal{Z})}}\varinjlim (\H_{\fa}^i(X)\otimes_RT)\\
&\cong \underset{{\fa\in F(\mathcal{Z})}}\varinjlim \H_{\fa T}^i(X\otimes_RT)\\
&\cong \  \  \underset{{{\widetilde{\fb}\in \Omega}}}\varinjlim \  \  \H_{\widetilde{\fb}}^i(X\otimes_RT)\\
&\cong \underset{{{\widetilde{\fb}\in F(\mathcal{Z}^f)}}}\varinjlim \H_{\widetilde{\fb}}^i(X\otimes_RT)\\
&\cong \  \  \H_{\mathcal{Z}^f}^i(X\otimes_RT),
\end{array}\]
which completes the proof of (ii).
\end{proof}

Let $\mathcal{Z}$ be a specialization closed subset of $\Spec R$. Let $S$ be a multiplicatively closed subset of
$R$ and $f:R\lo S^{-1}R$ be the natural ring homomorphism.  In this case, we denote $\mathcal{Z}^f$ by $S^{-1}\mathcal{Z}$.
Clearly, $$S^{-1}\mathcal{Z}=\left\{S^{-1}\fp|\  \fp\cap S=\emptyset ~\text{and}~ \fp\in\mathcal{Z}\right\}.$$  In particular, for a
prime ideal $\fp$ of $R$, we denote $(R-\fp)^{-1}\mathcal{Z}$ by $\mathcal{Z}_{\fp}$. Assume that $R$ is local with the
unique maximal ideal $\fm$ and $\hat{R}$ is the completion of $R$ with respect to the $\fm$-adic topology. Let $f:R\lo
\hat{R}$ be the natural ring homomorphism.  In this case, we denote $\mathcal{Z}^f$ by $\widehat{\mathcal{Z}}$. Restating
Lemma \ref{B} (ii) for the flat $R$-algebras $S^{-1}R$ and $\hat{R}$ yields the following result.

\begin{cor}\label{E} Let $\mathcal{Z}$ be a specialization closed  subset of $\Spec R$ and \emph{$X\in \mathrm{D}_{
\sqsubset}(R)$}. Then the following statements hold.
\begin{itemize}
\item[(i)] Assume that $S$ is a multiplicatively closed subset of $R$. There is a natural $S^{-1}R$-isomorphism \emph{$S^{-1}
(\H_{\mathcal{Z}}^i(X)) \cong \H_{S^{-1}\mathcal{Z}}^i(S^{-1}X)$} for all integers $i$.
\item[(ii)] Assume that $(R,\fm)$ is a local ring. There is a natural $\hat{R}$-isomorphism \emph{$\H_{\mathcal{Z}}^i(X)\otimes_R
\hat{R}\cong \H_{\widehat{\mathcal{Z}}}^i(X\otimes_R\hat{R})$} for all integers $i$.
\end{itemize}
\end{cor}

The next result provides a comparison between the annihilation of local cohomology modules with respect to a specialization
closed subset of $\Spec R$ and the annihilation of their localizations.

\begin{lem}\label{F} Let $\mathcal{Z}$ be a specialization closed subset of $\Spec R$ and $X\in \mathrm{D}^f_{\sqsubset}(R)$.
Then for every $\fa\in F(\mathcal{Z})$ and every integer $t$,  the following statements
are equivalent:
\begin{itemize}
\item[(i)]\emph{{There exists a positive integer $l$ such that $\fa^{l}\H_{\mathcal{Z}}^i(X)=0$ for all $i<t$.}}
\item[(ii)]\emph{{For every $\fp\in\Spec R$, there exists a positive integer $l_{\fp}$ such that $\fa^{l_{\fp}}
\H_{{\mathcal{Z}}_{\fp}}^i(X_{\fp})
=0$ for all $i<t.$ }}
\end{itemize}
\end{lem}

\begin{proof} (i)$\Rightarrow$(ii) immediately follows by Corollary \ref{E} (i).

(ii)$\Rightarrow$(i) Clearly, we may assume that $t\geq 1-\sup X$. We proceed by induction on $t$. Let $t=1-\sup X$. Then by Lemma
\ref{A} (i), $\H_{\mathcal{Z}}^{-\sup X}(X)$ is  finitely generated, and so we may assume that $$\Ass_{R}(\H_{\mathcal{Z}}^{-\sup
{X}}(X))=\left\{\frak{p_{1}},\ldots, \fp_{r}\right\}.$$ Now, by our assumption, there exist positive integers $l_{\frak{p_{1}}},
\ldots, l_{\fp_{r}}$ such that $$\fa^{l_{\fp_i}}\H_{{\mathcal{Z}}_{\fp_{i}}}^{-\sup {X}}(X_{\fp_i})=0$$ for all $i=1,\ldots,r$.
Let $l:=\max\left\{l_{\frak{p_{1}}}, \ldots,l_{\fp_{r}}\right\}$. Then, in view of Corollary \ref{E} (i), $(\fa^{l}
\H_{\mathcal{Z}}^{-\sup {X}}(X))_{\fp_{i}}=0$ for all $i=1, \ldots,r$. Thus $\fa^{l}\H_{\mathcal{Z}}^{-\sup {X}}(X)=0$, because
$$\Ass_{R}(\fa^l\H_{\mathcal{Z}}^{-\sup {X}}(X))\subseteq \Ass_{R}(\H_{\mathcal{Z}}^{-\sup {X}}(X)).$$ Hence,
$\fa^{l}\H_{\mathcal{Z}}^i(X)=0$ for all $i<1-\sup X$.

Next, suppose that $t>1-\sup X$ and the result has been proved for $t-1$.
From the induction hypothesis, we deduce that there exists a positive integer $l_1$ such that $\fa^{l_1}\H_{\mathcal{Z}}^i(X)=0$ for
all $i<t-1$. Then, Corollary \ref{D} yields that $\H_{\mathcal{Z}}^i(X)$ is finitely generated for all $i<t-1$. Now, Lemma \ref{C}
implies that $\Hom_R(R/\fa,\H_{\mathcal{Z}}^{t-1}(X))$ is finitely generated. By the assumption, for every prime ideal $\fp$ of $R$,
there exists a positive integer $l_{\fp}$ such that $\fa^{l_{\fp}}\H_{{\mathcal{Z}}_{\fp}}^{t-1}(X_{\fp})=0$, and so
$\Supp_R(\H_{\mathcal{Z}}^{t-1}(X)) \subseteq\V(\fa)$. Therefore, $$\Ass_{R}(\H_{\mathcal{Z}}^{t-1}(X)) =\Ass_R(\Hom_R(R/\fa,
\H_{\mathcal{Z}}^{t-1}(X)))$$ is finite. Hence, by a similar argument as in the case $t=1-\sup X$, we may find a positive integer
$l_2$ such that $\fa^{l_2}\H_{\mathcal{Z}}^{t-1}(X)=0$.  Finally, set $l:=\max\left\{l_1,l_2\right\}$.
\end{proof}

Let us come to the last preparation for proving the main result of this section.

\begin{lem}\label{F1} Let $\mathcal{Z}$ be a specialization closed subset of $\Spec R$ and $\fp\in \mathcal{Z}$. Let $\widetilde{\fb}$
be an ideal of the ring $R_{\fp}$. If $\widetilde{\fb}\in F(\mathcal{Z}_{\fp})$, then $\widetilde{\fb}\cap R\in F(\mathcal{Z})$.
\end{lem}

\begin{proof} Assume that $\widetilde{\fb}\in F(\mathcal{Z}_{\fp})$. Let $\widetilde{\fb}=\bigcap_{i=1}^nQ_i$  be a minimal primary
decomposition of $\widetilde{\fb}$ in $R_{\fp}$. Let $1\leq i\leq n$. As $\V(\widetilde{\fb})\subseteq \mathcal{Z}_{\fp}$, one has
$\sqrt{Q_i}\in \mathcal{Z}_{\fp}$, and so $$\sqrt{Q_i\cap R}=\sqrt{Q_i}\cap R\in \mathcal{Z}.$$ This completes the argument, because
$\widetilde{\fb}\cap R=\bigcap_{i=1}^n(Q_i\cap R)$.
\end{proof}

The following result is the derived category analogue of Faltings' Local-global Principle Theorem for a single specialization closed
subset $\mathcal{Z}$ of $\Spec R$.

\begin{thm}\label{G} Let $\mathcal{Z}$ be a specialization closed subset of $\Spec R$ and $X\in \mathrm{D}^f_{\sqsubset}(R)$.
Then for every integer $t$,  the following statements are equivalent:
\begin{itemize}
\item[(i)]\emph{{$\H_{\mathcal{Z}}^i(X)$ is a finitely generated $R$-module for all $i<t$.}}
\item[(ii)]\emph{{$\H_{\mathcal{Z}_{\fp}}^i(X_{\fp})$ is a finitely generated $R_{\fp}$-module for all $i<t$
and all $\fp\in\Spec R$.}}
\end{itemize}
\end{thm}

\begin{proof} (i)$\Rightarrow$(ii) is clear by Corollary \ref{E} (i).

(ii)$\Rightarrow$(i) We may and do assume that $t\geq 1-\sup X$ and proceed by induction on $t$. If $t=1-\sup X$, then by Lemma
\ref{A} (i) we see that $\H_{\mathcal{Z}}^{i}(X)$ is a finitely generated $R$-module for all $i<t$. Let $t>1-\sup X$ and suppose
that the result has been proved for $t-1$. The induction hypothesis implies that $\H_{\mathcal{Z}}^i(X)$ is a finitely generated
$R$-module for all $i<t-1$, and so by Lemma \ref{C} the $R$-module $\L_{\fb}:=\Hom_R(R/\fb,\H_{\mathcal{Z}}^{t-1}(X))$ is
finitely generated for all $\fb\in F(\mathcal{Z})$.

Fix $\fb\in F(\mathcal{Z})$. For every prime ideal $\fp$ of $R$, we set $$\fa_{\fp}:=\Ann_{R_{\fp}}(\H_{\mathcal{Z}_{\fp}}^{t-1}(X_{\fp}))
\cap R.$$ Lemma \ref{F1} yields that $\fa_{\fp}\in F(\mathcal{Z})$. As $\fa_{\fp}\L_{\fb}$ is a finitely generated $R$-module and
$(\fa_{\fp}\L_{\fb})_{\fp}=0$, there exists an element $x_{\fp}$ in $R-\fp$ such that
$(\fa_{\fp}\L_{\fb})_{x_{\fp}} =0$. Now, for every prime ideal $\fp$ of $R$, set $U_{x_{\fp}}:=\Spec R-\V(Rx_{\fp})$ and notice that
for every $\fq\in U_{x_{\fp}}$, one has $(\fa_{\fp}\L_{\fb})_{\fq}=0$. Since any increasing chain of open subsets of $\Spec R$ is
stationary, there exists a finite subset $\left\{\fp_1,...,\fp_{\ell}\right\}$ of $\Spec R$ such that $$\Spec R=\bigcup_{i=1}^{\ell}
U_{x_{\fp_i}}.$$ Hence, by setting ${\fa}:=\bigcap_{i=1}^{\ell}\fa_{\fp_{i}}$, one can see that $\fa\in F(\mathcal{Z})$ and
$(\fa\L_{\fb})_{\fp}=0$ for all $\fp\in \Spec R$. So,  $\fa\L_{\fb}=0$. This implies that $\fa (0:_{\H_{\mathcal{Z}}^{t-1}(X)}\fb)=0$,
because $\L_{\fb}\cong 0:_{\H_{\mathcal{Z}}^{t-1}(X)}\fb$. By Lemma \ref{A} (ii), one has $\Supp_R \H_{\mathcal{Z}}^{t-1}(X)
\subseteq \mathcal{Z}$. Hence $\H_{\mathcal{Z}}^{t-1}(X)=\Gamma_{\mathcal{Z}}(\H_{\mathcal{Z}}^{t-1}(X))$, and so $$\H_{\mathcal{Z}}^{t-1}(X)
=\underset{\fb\in F(\mathcal{Z})}\bigcup (0:_{\H_{\mathcal{Z}}^{t-1}(X)}\fb).$$ Thus
$\fa\H_{\mathcal{Z}}^{t-1}(X)=0$, and so $\H_{\mathcal{Z}}^{t-1}(X)\cong \Hom_R(R/\fa,\H_{\mathcal{Z}}^{t-1}(X))$. Now, Lemma \ref{C}
completes the proof.
\end{proof}

By Corollary \ref{D} and Theorem \ref{G}, one can immediately deduce the following result.

\begin{cor}\label{H}
Let $\mathcal{Z}$ be a specialization closed subset of $\Spec R$ and $X\in\mathrm{D}^f_{\sqsubset}(R)$.
\emph{{$$f_{\mathcal{Z}}(X)=\inf \left\{i\in \mathbb{Z}|\ \H_{\mathcal{Z}}^i(X) \  \text{is not
finitely generated}\right\}=\inf \left\{f_{\mathcal{Z_{\fp}}}(X_{\fp})|\  \fp \in \Spec R\right\}.$$}}
\end{cor}

\section{Faltings' Annihilator Theorem}

We start this section with the following technical, but useful, result.

\begin{lem}\label{I} Let $\mathcal{Z}\subseteq \mathcal{Y}$ be two specialization closed subsets of $\Spec R$ such that
\emph{$\V(\fp)\cap (\mathcal{Y}-\mathcal{Z})=\left\{\fp\right\}$} for all $\fp\in \mathcal{Y}-\mathcal{Z}$. Then for every
injective $R$-module $E$, there exists a natural $R$-isomorphism $$\Theta_{E}:\frac{\Gamma_{\mathcal{Y}}(E)}{\Gamma_{\mathcal{Z}}(E)}
\lo \underset{\tiny {\fp\in\mathcal{Y}-\mathcal{Z}}}\bigoplus\Gamma_{\fp R_{\fp}}(E_{\fp}).$$
\end{lem}

\begin{proof} Let $M$ be an $R$-module. There exists a natural $R$-homomorphism $$\theta_{M}:\Gamma_{\mathcal{Y}}(M)\longrightarrow
\underset{\tiny{\fp\in\mathcal{Y}-\mathcal{Z}}}\bigoplus \Gamma_{\fp R_{\fp}}(M_{\fp}),$$ with $\theta_{M}(m):=(\frac{m}{1})_{\fp}$
for all $m\in \Gamma_{\mathcal{Y}}(M)$. Let $m\in\Gamma_{\mathcal{Y}}(M)$. Then $\Supp_RRm\subseteq \mathcal{Y}$, and hence our
assumption on $\mathcal{Y}-\mathcal{Z}$ implies that each element of $(\Supp_RRm)\cap (\mathcal{Y}-\mathcal{Z})$ is minimal in
$\Supp_RRm$. Thus $\frac{m}{1}\in \Gamma_{\fp R_{\fp}}(M_{\fp})$ for all $\fp\in \mathcal{Y}-\mathcal{Z}$ and $(\Supp_RRm)
\cap (\mathcal{Y}-\mathcal{Z})$ is a finite set. So, $\theta_{M}$ is well-defined. Also, one can easily check that
$\Ker\theta_{M}=\Gamma_{\mathcal{Z}}(M)$.

Next, we prove that for any injective $R$-module $E$, the $R$-homomorphism $\theta_{E}$ is surjective. Let $$E=\underset{\tiny{\fq\in
\Spec R}}\bigoplus \E(R/\fq)^{(\mu^{0}(\fq,E))}$$ be an injective $R$-module. Let $\fp_{\circ}\in \mathcal{Y}-\mathcal{Z}$ and $\frac{x}{s}
\in \Gamma_{\fp_{\circ} R_{\fp_{\circ}}}(E_{\fp_{\circ}})$.  Since $x\in E$, $x=(x_{\fq})_{\fq},$ where $x_{\fq}\in \E(R/\fq)^{(\mu^{0}(\fq,
E))}$. As  $\frac{x}{s}\in \Gamma_{\fp_{\circ} R_{\fp_{\circ}}}(E_{\fp_{\circ}})$,  there is a positive integer $n$ and $\widetilde{s}
\in R-\fp_{\circ}$ such that $\widetilde{s}\fp_{\circ}^nx=0$. Let $\fq$ be a prime ideal of $R$ with $\fp_{\circ}\nsubseteq \fq$ and let
$t_{\fq}\in \fp_{\circ}-\fq$. Since $$\E(R/\fq)^{(\mu^{0}(\fq,E))}\overset{t_{\fq}^n}\longrightarrow \E(R/\fq)^{(\mu^{0}(\fq,E))}$$ is an
isomorphism, we get that $\widetilde{s}x_{\fq}=0$. Next, let $\fq$ be a prime ideal of $R$ with $\fq\nsubseteq \fp_{\circ}$. There is
a positive integer $n_{\fq}$ such that $\fq^{n_{\fq}}x_{\fq}=0$. Let $s_{\fq}\in \fq-\fp_{\circ}$. Then $s_{\fq}^{n_{\fq}}x_{\fq}=0$.
So, we may take $\check{s}\in R-\fp_{\circ}$ such that $\check{s}x_{\fq}=0$ for all $\fq\neq \fp_{\circ}$. Note that only finitely many
of $x_{\fq}\ $'s are nonzero. Thus, without loss of generality, we may assume that $x_{\fq}=0$ for all $\fq\neq \fp_{\circ}$. In particular, $(0:_Rx)=(0:_Rx_{\fp_{\circ}}).$ Hence $x$ is annihilated by some power of $\fp_{\circ}$, and so $x\in \Gamma_{\mathcal{Y}}(E)$.
On the other hand, using the fact
that the map $$\E(R/\fp_{\circ})^{(\mu^{0}(\fp_{\circ},E))}\overset{s}\longrightarrow \E(R/\fp_{\circ})^{(\mu^{0}(\fp_{\circ},E))}$$ is an
isomorphism, one deduces that $sy_{\fp_{\circ}}=x_{\fp_{\circ}}$ for some $y_{\fp_{\circ}}\in \E(R/\fp_{\circ})^{(\mu^{0}(\fp_{\circ},E))}$.
Let $\delta$ denote the Kronecker delta. Then for $y:=(\delta_{\fp_{\circ},\fq}y_{\fp_{\circ}})_{\fq}$, we have $\frac{y}{1}=\frac{x}{s}$
in $E_{\fp_{\circ}}$. Note that there exists a positive integer $t$ such that $\fp_{\circ}^t y=0$, and so $y\in \Gamma_{\mathcal{Y}}(E)$.
Since $\V(\fp_{\circ})\cap (\mathcal{Y}-\mathcal{Z})=\left\{\fp_{\circ}\right\}$, we deduce that $\frac{y}{1}=0$ in $E_{\fp}$ for all $\fp\in
(\mathcal{Y}-\mathcal{Z})-\left\{\fp_{\circ}\right\}$. So, $$\theta_{E}(y)=(\frac{y}{1})_{\fp}=(\delta_{\fp_{\circ},\fp}\frac{x}{s})_{\fp}.$$
Therefore $\theta_{E}$ is surjective, and so it induces a natural $R$-isomorphism $$\Theta_{E}:\frac{\Gamma_{\mathcal{Y}}(E)}{\Gamma_{
\mathcal{Z}}(E)}\lo \underset{\tiny{\fp\in \mathcal{Y}-\mathcal{Z}}}\bigoplus\Gamma_{\fp R_{\fp}}(E_{\fp}).$$
\end{proof}

Next, we establish a useful long exact sequence of local cohomology modules.

\begin{lem}\label{J} Let $\mathcal{Z}\subseteq \mathcal{Y}$ be two specialization closed subsets of $\Spec R$ such that \emph{$\V(\fp)
\cap (\mathcal{Y}-\mathcal{Z})=\left\{\fp\right\}$} for all $\fp\in \mathcal{Y}-\mathcal{Z}$. Then for any $X\in \mathrm{D}_{\sqsubset}(R),$ there is a
long exact sequence \emph{$$\cdots \tiny{\rightarrow} \underset{\tiny{\fp\in \mathcal{Y}-\mathcal{Z}}}\bigoplus \large{\H}_{\fp R_{\fp}}^{i-1}(X_{\fp})
\tiny{\rightarrow}  \large{\H}_{\mathcal{Z}}^i(X)\rightarrow \large{\H}_{\mathcal{Y}}^i(X)\tiny{\rightarrow}  \underset{\tiny{\fp\in \mathcal{Y}-\mathcal{Z}}}\bigoplus
\large{\H}_{\fp R_{\fp}}^i(X_{\fp})\tiny{\rightarrow}  \large{\H}_{\mathcal{Z}}^{i+1}(X)\tiny{\rightarrow} \cdots.$$}
\end{lem}

\begin{proof} First, let $M$ be an $R$-module, and set $\Gamma_{\mathcal{Y}/\mathcal{Z}}(M):=\Gamma_{\mathcal{Y}}(M)/\Gamma_{\mathcal{Z}}(M)$.
Note that $\Gamma_{\mathcal{Y}/\mathcal{Z}}(-)$ is a functor from the category of $R$-modules to itself, but not necessarily left exact. We
consider the right derived functor of this functor in $\mathrm{D}(R)$. Let $X\in \mathrm{D}_{\sqsubset}(R)$ and $I$ be an injective resolution
of $X$.  Then for every prime ideal $\fp$ of $R$, we may check that $I_{\fp}$ is an injective resolution of the $R_{\fp}$-complex $X_{\fp}$.

We set
$\H_{\mathcal{Y}/ \mathcal{Z}}^i(X):=\H_{-i}(\Gamma_{\mathcal{Y}/\mathcal{Z}}(I))$ for all integers $i$. One can use the following exact
sequence of $R$-complexes $$0\longrightarrow\Gamma_{\mathcal{Z}}(I)\longrightarrow\Gamma_{\mathcal{Y}}(I)\longrightarrow \Gamma_{\mathcal{Y}/
\mathcal{Z}}(I)\longrightarrow 0,$$ to obtain the long exact sequence $$\cdots \longrightarrow\H_{\mathcal{Y}/\mathcal{Z}}^{i-1}(X)\longrightarrow
\H_{\mathcal{Z}}^i(X)\longrightarrow \H_{\mathcal{Y}}^i(X)\longrightarrow\H_{\mathcal{Y}/\mathcal{Z}}^i(X)\rightarrow \H_{\mathcal{Z}}^{i+1}(X)
\rightarrow \cdots.$$ Lemma \ref{I} yields that the two complexes $\Gamma_{\mathcal{Y}/\mathcal{Z}}(I)$ and $\underset{\tiny{\fp\in \mathcal{Y}-
\mathcal{Z}}}\bigoplus\Gamma_{\fp R_{\fp}}(I_{\fp})$ are isomorphic, and so $$\H_{\mathcal{Y}/\mathcal{Z}}^i(X)\cong \underset{\tiny{\fp\in
\mathcal{Y}-\mathcal{Z}}}\bigoplus \H_{\fp R_{\fp}}^i(X_{\fp})$$ for all integers $i$. This completes the proof.
\end{proof}

Next, we include the following immediate consequence.

\begin{cor}\label{K} Let $\mathcal{Z}$ be a specialization closed subset of $\Spec R$ and $X\in \mathrm{D}_{\sqsubset}(R)$.
Then the following statements hold.
\begin{itemize}
\item[(i)] For every integer $n$,  \emph{$\mathcal{Z}_n:=\left\{~\fp\in\mathcal{Z}|~\h\fp\geq n\right\}$} is a specialization
closed subset of $\Spec R$ and \emph{$\bigcap_{n\in\mathbb{Z}}\mathcal{Z}_{n}=\emptyset$.}
\item[(ii)] If $\dim R$ is finite, then \emph{$\H_{\mathcal{Z}_{n}}^i(X)=0$} for all $i$ and all $n> \dim R$.
\item[(iii)] For any two integers $i$ and $n$, there exists an exact sequence \emph{$$\H_{\mathcal{Z}_{n+1}}^i(X)
\longrightarrow \H_{\mathcal{Z}_{n}}^i(X)\longrightarrow\underset{\tiny{\fp\in \mathcal{Z}_{n}-\mathcal{Z}_{n+1}}}\bigoplus
\H_{\fp R_{\fp}}^i(X_{\fp}).$$}
\end{itemize}
\end{cor}

We need to apply the following first quadrant spectral sequence in the proof of the main result of this section.

\begin{lem}\label{L} Let $\mathcal{Z}$ be a specialization closed subset of $\Spec R$. Then for any $X\in \mathrm{D}_{\sqsubset}(R)$
with $\sup X=0$ and any $\fa\in F(\mathcal{Z})$, there is a first quadrant spectral sequence \emph{$$\E_2^{p,q}:=\H_{\fa}^{p}
(\H_{\mathcal{Z}}^{q}(X))\underset{p}\Longrightarrow \H_{\fa}^{p+q}(X).$$}
\end{lem}

\begin{proof} By \cite[1.6]{FI}, we have the following spectral sequence $$\E^2_{p,q}:=\H_{\fa}^{-p}(\H_{q}(\ugamma_{\mathcal{Z}}(X)))
\underset{p}\Longrightarrow \H_{\fa}^{-p-q}(\ugamma_{\mathcal{Z}}(X)).$$ Let $I$ be
an injective resolution of $X$. Then, one has the following natural $R$-isomorphisms
\[\begin{array}{rl}
\H_{\fa}^{-p-q}(\ugamma_{\mathcal{Z}}(X))&\cong \H_{\fa}^{-p-q}(\Gamma_{\mathcal{Z}}(I))\\
&\cong \H_{p+q}(\ugamma_{\fa}(\Gamma_{\mathcal{Z}}(I)))\\
&\cong \H_{p+q}(\Gamma_{\fa}(\Gamma_{\mathcal{Z}}(I)))\\
&\cong \H_{p+q}(\Gamma_{\fa}(I))\\
&\cong \H_{p+q}(\ugamma_{\fa}(X))\\
&\cong \H_{\fa}^{-p-q}(X),
\end{array}\]
which completes the argument.
\end{proof}

Now, we are ready to prove the Annihilator Theorem for local cohomology modules of complexes.

\begin{thm}\label{M} Let $\mathcal{Z}\subseteq\mathcal{Y}$ be two specialization closed subsets of $\Spec R$,
$X\in\mathrm{D}^f_{\Box}(R)$ and $n$ an integer. Consider the following statements:
\begin{itemize}
\item[(i)]\emph{{There exists an ideal $\fa\in F(\mathcal{Y})$ such that $\fa\H_{\mathcal{Z}}^i(X)=0$ for all $i\leq n.$ }}
\item[(ii)]\emph{{For every $\fq\notin \mathcal{Y}$ and every $\fp\in\mathcal{Z}\cap \V(\fq)$, one has $\depth_{R_{\fq}}X_{\fq}+
\h\fp/\fq> n.$}}
\end{itemize}
Then, \emph{(i)} always implies \emph{(ii)} and \emph{(ii)} implies \emph{(i)}, provided $R$ is a homomorphic image of a
finite-dimensional Gorenstein ring.
\end{thm}

\begin{proof} Note that by Lemma \ref{A} (i), $\H_{\mathcal{Z}}^i(X)=0$ for all $i<-\sup X$ and $\H_{\mathcal{Z}}^{-\sup X}(X)$ is
finitely generated. Set $T:=\Sigma^{-\sup X}X$ and note that $\fa\H_{\mathcal{Z}}^i(X)=0$ for all $i\leq n$  if and only if
$\fa\H_{\mathcal{Z}}^i(T)=0$ for all $i\leq n+\sup X$. Also, one can see that $$\depth_{R_{\fq}}T_{\fq}=\depth_{R_{\fq}} X_{\fq}+
\sup X$$for all $\fq\in\Spec R$. Therefore, by replacing $X$ with $T$, we may and do assume that $\sup X=0$, and so
$\H_{\mathcal{Z}}^i(X)=0$ for all $i<0$ and $\H_{\mathcal{Z}}^{0}(X)$ is a finitely generated $R$-module.

(i)$\Rightarrow$(ii) First, we reduce the situation to the case that $R$ is a Gorenstein local ring.  To this
end, let $\fq\notin \mathcal{Y}$ and $\fp\in\mathcal{Z}\cap \V(\fq)$. We should show that $$\depth_{R_{\fq}}X_{\fq}+\h\fp/\fq> n,$$
which it is equivalent to show that $$\depth_{({R_{\fp})}_{\fq R_{\fp}}}(({X_{\fp}})_{\fq R_{\fp}})+\dim R_{\fp}/\fq R_{\fp}> n.$$
Hence, in view of Corollary \ref{E} (i), by replacing $R$, $\mathcal{Y}$,  $\mathcal{Z}$ and $X$ with $R_{\fp}$, $\mathcal{Y}_{\fp}$,
$\mathcal{Z}_{\fp}$ and $X_{\fp}$; respectively, we may and do assume that $(R,\fm)$ is a local ring and we must show that for every
$\fq\notin \mathcal{Y}$, $$\depth_{R_{\fq}}X_{\fq}+\dim R/\fq>n.$$ Let $\hat{R}$ be the completion of $R$ with respect to the
$\fm$-adic topology.  We have $$\dim R/\fq=\dim \hat{R}/\fq\hat{R}=\dim\hat{R} /\frak Q,$$ for some $\frak Q\in\Min{\fq\hat{R}}$,
and one can see that $\frak Q\notin \widehat{\mathcal{Y}}.$ On the other hand, in view of \cite[Corollary 2.6]{Iy} we have
$$\depth_{R_{\fq}} X_{\fq}=\depth_{\hat{R}_{\frak{Q}}}(X_{\fq}\otimes_{R_{\fq}}\hat{R}_{\frak{Q}})=\depth_{\hat{R}_{\frak{Q}}}(
X\otimes_{R}{\hat{R}})_{\frak{Q}}.$$ Corollary \ref{E} (ii) yields that there is an $\hat{R}$-isomorphism $\H_{\mathcal{Z}}^i(X)
\otimes_{R}\hat{R}\cong \H_{\widehat{\mathcal{Z}}}^i(X\otimes_{R}\hat{R})$ for all integers $i$. Thus we may assume that $R$ is a
complete local ring, and so it is a homomorphic image of a Gorenstein local ring. Next, in view of Lemma \ref{B} (i), we can assume
that $R$ is a Gorenstein local ring.

Now for a given $\fq\notin\mathcal{Y},$ we should show that $\H_{\fq R_{\fq}}^j(X_{\fq})=0$ for all $j\leq n-\dim R/\fq.$
By Lemma \ref{L}, we have a first quadrant spectral sequence $$\E_2^{p,q}:=\H_{\fm}^{p}(\H_{\mathcal{Z}}^{q}(X))\underset{p}\Longrightarrow
\H_{\fm}^{p+q}(X).$$ So, we can use our assumption to deduce that $\fa^t\H_{\fm}^{j+\dim R/\fq}(X)=0$ for some positive
integer $t$. Hence, by using the Local Duality Theorem \cite[Chapter V, Theorem 6.2]{Ha}, ${\fa^t}\Ext_{R}^{\dim R-j-\dim R/\fq}(X,R)=0$,
and so $$\Supp_R(\Ext_R^{\dim R-j-\dim R/\fq}(X,R))\subseteq \V(\fa)\subseteq \mathcal{Y}.$$ As $\fq\notin \mathcal{Y}$ and $\dim R_{\fq}
=\dim R-\dim R/\fq,$ we deuce that $\Ext_{R_{\fq}}^{\dim R_{\fq}-j}(X_{\fq}, R_{\fq})=0$. Therefore, the Local Duality Theorem implies
that $\H_{\fq R_{\fq}}^{j}(X_{\fq})=0$.

(ii)$\Rightarrow$(i) First, note that by Lemma \ref{B} (i) we may assume that $R$ is Gorenstein and $\dim R<\infty$. Fix a non-negative
integer $i\leq n$. Let $\fq\notin\mathcal{Y}$ and $\fp\in\mathcal{Z}\cap \V(\fq)$. By the assumption, we have $$i-\h\fp/\fq\leq n-\h\fp/\fq<
\depth_{R_{\fq}}X_{\fq},$$ and so $\H_{\fq R_{\fq}}^{i-\h\fp/\fq}(X_{\fq})=0$. Then the Local Duality Theorem yields that
$\Ext_{R}^{\h\fp-i}(X,R)_{\fq}=0$. Thus, one deduces that $\Supp_{R_{\fp}}(\Ext_{R_{\fp}}^{\h\fp-i}(X_{\fp},R_{\fp}))
\subseteq \mathcal{Y}_{\fp}$.

Next, for every $\fp\in \mathcal{Z}$, set $${\fa}_{\fp,i}:=\Ann_{R_{\fp}}(\Ext_{R_{\fp}}^{\h\fp-i}(X_{\fp},
R_{\fp}))\cap R.$$ Applying Lemma \ref{F1} implies that $\V({\fa}_{\fp,i})\subseteq\mathcal{Y}$. Let $t$ be a non-negative integer and
$\fp\in \mathcal{Z}_{t}-\mathcal{Z}_{t+1}$. As ${\fa}_{\fp,i}\Ext_{R_{\fp}}^{t-i}(X_{\fp},R_{\fp})=0,$ there exists $x_{\fp}\in R-\fp$
such that $({\fa}_{\fp,i}\Ext_{R}^{t-i}(X,R))_{x_{\fp}}=0$. Now, set $U_{x_{\fp}}:=\Spec R-\V(Rx_{\fp})$ and note that for every
$\fq\in U_{x_{\fp}}$, one has $${\fa}_{\fp,i}\Ext_{R_{\fq}}^{t-i}(X_{\fq},R_{\fq})=0.$$ By using the fact that every increasing chain of
open subsets of $\Spec R$ is stationary, one can deduce that there exists a finite subset $\mathcal{W}_t$ of $\mathcal{Z}_{t}-\mathcal{Z}_{t+1}$
such that $$\underset{\fp\in \mathcal{Z}_{t}-\mathcal{Z}_{t+1}}\bigcup U_{x_{\fp}}=\underset{\fp\in \mathcal{W}_t}\bigcup U_{x_{\fp}}.$$ Hence,
by setting ${\fa}_{t,i}:=\underset{\fp\in \mathcal{W}_t}\bigcap\fa_{\fp,i}$, one can see that $\V({\fa}_{t,i})\subseteq\mathcal{Y}$ and
$\fa_{t,i}\Ext_{R_{\fp}}^{t-i}(X_{\fp},R_{\fp})=0$ for all $\fp\in\mathcal{Z}_{t}-\mathcal{Z}_{t+1}$. Applying the Local Duality Theorem
again implies that $\fa_{t,i}\H_{\fp R_{\fp}}^{i}(X_{\fp})=0$ for all $\fp\in\mathcal{Z}_{t}-\mathcal{Z}_{t+1}$. Now, using the exact sequence
given in Corollary \ref{K} (iii) implies that for the ideal $\fa_i:=\prod_{t=0}^{\dim R}\fa_{t,i}$, we have $\fa_i\H_{\mathcal{Z}}^{i}(X)=0$,
and so by setting $\fa:=\bigcap_{i=0}^n\fa_{i}$ the assertion follows.
\end{proof}

We close the paper with the following result.

\begin{cor}\label{P} Let $\mathcal{Z}\subseteq\mathcal{Y}$ be two specialization closed subsets of $\Spec R$ and $X\in\mathrm{D}^f_{\Box}(R)$.
Then the following statements hold.
\begin{itemize}
\item[(i)] $f_{\mathcal{Z}}^{\mathcal{Y}}(X)
\leq \lambda_{\mathcal{Z}}^{\mathcal{Y}}(X)$.
\item[(ii)] Assume that $R$ is a homomorphic image of a finite-dimensional Gorenstein ring. Then $f_{\mathcal{Z}}^{\mathcal{Y}}(X)=
\lambda_{\mathcal{Z}}^{\mathcal{Y}}(X)$ and $f_{\mathcal{Z}}^{\mathcal{Y}}(X)
=\inf \left\{f_{\mathcal{Z_{\fp}}}^{\mathcal{Y_{\fp}}}(X_{\fp})|\   \fp \in \Spec R\right\}$.
\end{itemize}
\end{cor}

\begin{proof} (i) follows by the implication (i)$\Longrightarrow$(ii) in Theorem \ref{M}.

(ii) The first assertion of (ii) follows by Theorem \ref{M}.

Denote $\inf \left\{f_{\mathcal{Z_{\fp}}}^{\mathcal{Y_{\fp}}}(X_{\fp})|\   \fp \in \Spec R\right\}$ by $t$. For every prime ideal
$\fp$ of $R$, Corollary \ref{E} (i) easily yields that $f_{\mathcal{Z}}^{\mathcal{Y}}(X)\leq f_{\mathcal{Z_{\fp}}}^{\mathcal{Y_{\fp}}}
(X_{\fp})$, and so $f_{\mathcal{Z}}^{\mathcal{Y}}(X)\leq t$. Let $n$ be any integer with $n<t$ and $\fp$ be a prime ideal of $R$.
As $n<f_{\mathcal{Z_{\fp}}}^{\mathcal{Y_{\fp}}}(X_{\fp})$, it turns out that there exists $\fc\in F(\mathcal{Y}_{\fp})$ such that
$\fc\H_{{\mathcal{Z}}_{\fp}}^i(X_{\fp})=0$ for all $i\leq n$.  Hence, by Theorem \ref{M}, for every $\fq\notin\mathcal{Y}$
and every $\fp\in \mathcal{Z}\cap \V(\fq)$, one can deduce that$$\depth_{R_{\fq}} X_{\fq}+\h{\fp}/{\fq}> n.$$ Thus, we can apply
Theorem \ref{M} again to deduce that there exists $\fa\in F(\mathcal{Y})$ such that $\fa\H_{{\mathcal{Z}}}^i(X)=0$ for all $i\leq n$.
Therefore $f_{\mathcal{Z}}^{\mathcal{Y}}(X)>n$, and so $f_{\mathcal{Z}}^{\mathcal{Y}}(X)\geq t$.

\end{proof}

%%%%%%%%%%%%%%%%%%%%%%%%%%%%%%%%%%%%%%%%%%%%%%%%%%%

\end{document}